\documentclass{amsart}

\usepackage[utf8]{inputenc}
\usepackage[T1]{fontenc}
\usepackage{amsmath}
\usepackage{amsfonts}
\usepackage{amssymb}
\usepackage{graphicx}
\usepackage{color}
\usepackage{ulem}
\usepackage{selinput}
\usepackage{caption}
\usepackage{subcaption}
\usepackage{cancel}
\usepackage{enumitem}
\usepackage[paperwidth=7in, paperheight=10in, margin=0.8in]{geometry}
\usepackage{todonotes}
\usepackage{stmaryrd}
\usepackage{tikz}
\usepackage{lipsum}

\usepackage{enumitem}
\usepackage{subcaption}

\usepackage[style=numeric,giveninits=true,sorting=none,maxbibnames=6,citestyle=numeric-comp,isbn=false]{biblatex}
\renewbibmacro{in:}{%
  \ifentrytype{article}{}{\printtext{\bibstring{in}\intitlepunct}}}
\renewbibmacro*{volume+number+eid}{%
  \printfield{volume}%
  \setunit*{\adddot}
\setunit*{\addnbspace}
  \printfield{number}%
  \setunit{\addcomma\space}%
 \printfield{eid}}
\DeclareFieldFormat[article]{number}{(#1)}
\AtEveryBibitem{\clearlist{language}}

\usetikzlibrary{calc}

\def\r{\rho}
\newcommand{\Ri}{{R_i}}
\newcommand{\Rj}{{R_j}}
\newcommand{\ri}{{\r_i}}
\newcommand{\rj}{{\r_j}}

\newcommand{\Rit}{{R_i^*}}
\newcommand{\Rjt}{{R_j^*}}

\newcommand{\rjt}{{\r_j^*}}
\newcommand{\s}{{\sigma^2}}
\newcommand{\si}{{\sigma_i}}
\newcommand{\sj}{{\sigma_j}}

\newcommand{\fe}{f_\var}
\newcommand{\Rem}{{\mathcal{R}_\var}}
\newcommand{\Remm}{{\tilde{\mathcal{R}}_\var}} 

\newcommand{\su}{{\operatorname{supp}}}

\def\var{\gamma}
\def\r{\rho}

\def\t{\tau}

\def\R{\mathbb{R}}
\def\RR{\mathbb{R}^2}
\def\RRR{\mathbb{R}^3}
\def\F{\mathcal{F}}

\def\E{\mathcal{E}}
\def\H{\mathcal{H}}

\newcommand{\der}{\frac{\partial}{\partial r}}

\newcommand{\derr}{\frac{\partial}{\partial r^2}}

\newcommand{\dersrs}{\frac{\partial^2}{\partial r'^2}}

\newcommand{\D}{{\Omega}}
\newcommand{\rik}{{\r_{i_k}}}

\newcommand{\ti}{{\theta_i}}
\newcommand{\tj}{{\theta_j}}

\newcommand{\Sij}{{S_{ij}}}

\newcommand{\Var}{{\operatorname{Var}}}

\theoremstyle{remark}

\author{Düring Bertram$^{(1)}$}
\address[1]{University of Warwick, Mathematics Institute, Gibbet Hill Road, CV47AL Coventry, UK}
\author{Fischer Michael$^{(2)}$}
\address[2]{University of Vienna, Fakultät für Mathematik, Oskar-Morgenstern-Platz 1, 1090 Vienna, Austria}
\author{Wolfram Marie-Therese$^{(1)}$}

\title[ELO]{An Elo-type rating model for players and teams of variable strength}
\date{}

\newtheorem{theorem}{Theorem}

\makeindex

\begin{document}
\maketitle

\begin{abstract}
The Elo rating system, which was originally proposed by Arpad Elo for chess, has become one of the most important rating systems in sports, economics and gaming nowadays. Its original formulation is based on two-player zero-sum games, but it has been adapted for team sports and other settings.\\
 In 2015, Junca and Jabin proposed a kinetic version of the Elo model, and showed that under certain assumptions the ratings do converge towards the players' strength. In this paper we generalise their model to account for variable performance of individual players or teams. We discuss the underlying modelling assumptions, derive the respective formal mean-field model and illustrate the dynamics with computational results.  
\end{abstract}

\section{Introduction}

Rating systems have become an indispensable tool to rank unobservable quantities, such as a players' strength based on observations, for example outcomes of games. Rating models were originally developed for sports; but are nowadays also used in gaming and financial markets. The Elo-rating system \cite{b:elo}  is one of the most prominent rating systems -- it is used in chess and other two-player zero sum games. Versions of the Elo-rating have been adopted for many other sports, for example basketball and football, see \cite{Forbes, i:nba}.  Other prominent rating systems include the Glicko rating system or Trueskill, see \cite{a:glicko98, a:trueskill06}.
Elo and Glicko are based on two-player zero sum games (here a player can be a single individual or an entire team), while Trueskill  is used in multi-player situations, as for
example in online gaming, see \cite{Minka2018TrueSkill2A, a:trueskill06}.

Elo himself tried to confirm the validity of the proposed rating
system using statistical experiments \cite{b:elo}. It was not until
2014 that Jabin and and Junca \cite{a:jabin} showed the convergence of
ratings towards the players' strength for a  continuous kinetic
version of the model. Junca \cite{junca:hal-03286591} later analysed the convergence of discrete ratings
in Robin-round tournament, in which players compete against all others in a round and discrete
ratings are updated after each such round.
 However, in this model the players' strength did
not change in time.  D\"uring et al. proposed a generalisation in
\cite{bib:During:elo}, in which players improve and loose skills based on the
outcome of games as well as daily performance fluctuations. A simpler
but related learning mechanism was proposed by Krupp in
\cite{mt:K2016}.

Kinetic models have been used very successfully to describe the behaviour of large interacting agent
systems in economics and social sciences. In all these applications interactions between agents -- such as encounters in games,  the trading of goods or the exchange of opinions -- are modelled via binary `collisions'. 
Toscani \cite{a:06toscani} was the first to introduce kinetic models in the context of opinion formation.
His ideas were later generalised for more complex opinion dynamics \cite{zbMATH06977672,zbMATH07146312,zbMATH06665924,zbMATH06875734,albi2016recent,albi2019boltzmann, bib:During:strongleaders, bib:During:inhomogeneous}, or in the context of wealth distribution \cite{bib:During:economykineticmodel, pareschi2014wealth} or knowledge growth in societies
\cite{BLW2016, BLW2016_2}. For a general overview on interacting
multi-agent systems and kinetic equations we refer to the book of
Pareschi and Toscani \cite{b:PLT13}.

The kinetic formulation of the Elo-model by Jabin and Junca assumes that each player is characterised by a constant strength $\rho$ (being an unobservable quantity) and a rating $R$, which changes based on the outcome of games. After each match between player $i$ and $j$ their respective ratings $R_i$ and $R_j$ are updated as follows
\begin{align}\label{eq:eloorig}
\begin{aligned}
\Rit &= \Ri + \var (S_{ij}- b(R_i-R_j)),\\
\Rjt &= \Rj + \var (-S_{ij} - b(R_j-R_i)).
\end{aligned}
\end{align}
Here  $b$ is an odd, monotone, increasing function, usually chosen as
$b(z) = \tanh(\nu z)$ with a scaling constant $\nu \in \R^+$. The
parameter $\var$ controls the speed of adjustment. The
outcome of the game is given by the random variable $\Sij$, which
takes the values $\lbrace -1,1\rbrace$, corresponding to a win or loss
(other, more fine grained outcomes like a tie can be added in a natural way). It is assumed to equal the expectation of $b(\rho_i-\rho_j)$, that is
\begin{align*}
\langle S_{ij}\rangle= \langle b(\ri-\rj)\rangle,
\end{align*}
 where $\ri,\rj$ are the underlying unobservable players'
 strength. Note that the interactions \eqref{eq:eloorig} are invariant
 with respect to translations and both the rating update \eqref{eq:eloorig} and the expected game outcome $\Sij$ depend only on the difference in $\rho$ and $R$, respectively, so these variables are defined on $\mathbb{R}.$
 
 Jabin and Junca then derived the corresponding  macroscopic model for the distribution of players $f(t,r,\rho)$ with respect to their rating $r$ and their strength $\rho$:
\begin{align}\label{eq:origelo}
\frac{\partial }{\partial t} f(t,\rho,r)+\frac{\partial }{\partial r}(a[f] f(t,\rho,r))&=0
\end{align}
with $$a[f] =\int\limits_{\mathbb{R}^3} w (r-r')
\big( b(\rho-\rho' )-b(r-r')\big)f(t,\rho',r') d\rho'  d r'$$ and initial condition $f(0,\rho,r) = f^0(\rho, r)$.
Here, the even probability distribution $w$ was introduced, to account for ranking dependent pairings in tournaments. If $w \equiv 1$ we consider a so-called \textit{all-play-all} game. If $w$ has compact support  only teams with close ratings compete. Possible choices for $w$ are
\begin{align}
w(r-r') &=e^{\frac{\log 2}{1+(r-r')^2}}-1 \;\text{ or }\; w(r-r') = \chi_{\{\lvert r-r'\rvert \leq c\}}. 
\end{align}
where $\chi$ denotes the indicator function (or smoothed variants
thereof) and $c>0$ is the maximal rating difference between paired competitors. If  $w>0$ Jabin and Junca \cite{a:jabin} showed that solutions to \eqref{eq:origelo} concentrate on the diagonal, providing the proof that the ratings indeed converge to the underlying strength.\\

\noindent In this work we propose a generalisation of the Elo-model
for teams of players with fluctuating strengths.
Our main contributions are the following
\begin{itemize}[nosep]
\item We propose and analyse an Elo-rating for teams, which includes
  stochastic variations in the team strength due to changes in the player setup.
\item We formally derive the respective Fokker-Planck equations and analyse their behaviour for long times.
\item We investigate the behaviour of solutions in the special case of
  competing teams whose players' strengths are distributed with a similar variance.
\item We illustrate the behaviour of the micro- and macroscopic models with computational experiments, consolidating and extending the analytical results.
\end{itemize}

This work is organised as follows: we propose a microscopic
generalisation of the well-known Elo-rating to teams of players and
illustrate the behaviour with microscopic simulations in
Section~\ref{re:microteams}. Section \ref{re:macroteams} focuses on
the corresponding formally derived macroscopic model and its
analysis. Next we investigate the model in the case of homogeneous
teams in Section \ref{s:scalinglimits} and report results of
computational experiments. Section~\ref{re:conclusion} concludes.

\section{A microscopic Elo-type rating for teams}\label{re:microteams}

We start by proposing a microscopic version of the Elo-rating for players with variable strength, which can also be used in the context of teams.

\subsection{Performance variations in teams and individual players}
In the following we consider a microscopic model, which accounts for performance fluctuations in teams as well as individuals. 
These fluctuations may be caused by varying individual or team
performance  (due to different line-ups) or for example by card
luck. We recall that small performance fluctuations in the individual
strength $\rho$ were modelled in \cite{bib:During:elo} by stochastic fluctuations in the strength. We will follow a different approach and replace the constant strength $\rho$ by a random variable $\lambda_\rho$ defined on a set of possible outcomes $\Omega_\rho$, which can be a finite set as well as an interval. This then allows us to define the stochastic process $\{ \lambda_\rho (t)\}_{t\in \R^+}$, whose expected value and variance will be denoted by
$\theta := \langle \lambda_\rho \rangle,$ and $\sigma^2 := \Var[\lambda_\rho],$ respectively.

We consider competing teams $T_i$, $T_j$ instead of
individual players. The corresponding expected value $\theta$ can be
interpreted as the mean strength of the team with a chosen line-up,
i.e.\ a subset of the team's players who will be playing in a particular game.  We assume similar as in \cite{a:jabin}, that the expected outcome of
the game between $T_i,T_j$ depends on the difference of teams' strengths through $b$:
\begin{align}\label{eq:boundedSij}
\langle S_{ij}\rangle= \langle b(\lambda_\ri-\lambda_\rj)\rangle.
\end{align}
If two teams $T_i$ and $T_j$ with ratings $\Ri$ and $\Rj$ meet, their ratings and strength after the game can be updated using  again 
\eqref{eq:eloorig}
where $\var$ is a scaling constant controlling the speed of adjustment. It is
usually chosen much smaller than the rating scores, in the hope that a player's rating slowly converges to its underlying strength. As discussed in the introduction we make the following assumption on $b$:
\begin{enumerate}[label=($B$)]
\item The function $b$ is  $C^3(\mathbb{R})$, monotonically increasing, bounded, odd and Lipschitz. \label{as:B}
\end{enumerate}
Since $b$ is non-linear the expected value and $b$ cannot be
interchanged. To calculate the expected value of $S_{ij}$ in \eqref{eq:boundedSij} we use the
\textit{Law of the unconscious statistician}
\cite{book:blitzstein2014introduction}, e.g.\ in the discrete case we obtain
\begin{align}\label{eq:lawofstat}
\langle S_{ij} \rangle  =
\sum_{x_i\in\Omega_\ri} \sum_{x_j\in\Omega_\rj} b(x_i-x_j) \tilde p(x_i, x_j) =
\sum_{x_i\in\Omega_\ri} \sum_{x_j\in\Omega_\rj} b(x_i-x_j) p(x_i)p(x_j) ,
\end{align} 
where $\tilde p$ denotes the probability of a possible line up $x_i$ playing against a line up $x_j$ and the second equality holds if this happens independently of each other. In the following we always assume this independence of the stochastic processes for team $T_i$ and $T_j$.

We can Taylor-expand $\langle S_{ij} \rangle $ in \eqref{eq:lawofstat} as described, for example, in  \cite{book:taylorprob}. Since we have $\langle \lambda_\ri-\lambda_\rj \rangle=\ti-\tj$ and $\Var [\lambda_\ri-\lambda_\rj]=\si^2+\sj^2$, it follows:
\begin{align}\label{eq:taylorE}
\begin{aligned}
\langle S_{ij} \rangle =\langle b\big( \lambda_\ri-\lambda_\rj\big) \rangle & \approx b(\ti-\tj)+\frac{1}{2}b''(\ti-\tj)(\si^2+\sj^2) \\
& =:b(\ti-\tj)+K(\ti-\tj,\si ,\sj ).
\end{aligned}
\end{align}
Note that the function $K$ is odd in the first argument and even in the other two. Similarly, the following holds for the variance
\begin{align}\label{eq:taylorV}
\Var [S_{ij}] \approx \big(b'(\ti-\tj) \big)^2(\si^2+\sj^2).
\end{align}

\subsection{Microscopic simulations}\label{subsec:nummericsmicro}
In the following we will illustrate the behaviour of the microscopic
model with various simulations. We consider $N$ teams $T_1 \ldots
T_N$; each team has $M$ players with strengths $\rik$ from which $m <
M$ distinct players are selected as line-up for each match. Let $\vec{\rho}_i=(\rho_{i_1},\dots ,\rho_{i_M})$  denote
the vector of all players in a team $T_i$. We assume without loss of
generality that the vector $\vec{\rho}_i$ is ordered.

Let us consider first the case that the $m$ players for the
line-up are chosen from the set of $M$ players uniformly. This can be done by generating $\binom{M}{m}$  normalised vectors $\vec{\lambda}_{i_k} \in \{ 0,1/m\}^M$ with $|\vec{\lambda}_{i_k}|=1$. Then the stochastic process $\lambda^t$ selects each vector $\vec{\lambda}_{i_k}$ with equal probability and we have
\begin{align*}
P(\Sij =1)=b(\vec{\lambda}_{i_k}\cdot \vec{\ri} - \vec{\lambda}_{j_k}\cdot \vec{\rj} ),
\end{align*}
for a match between $T_i$ and $T_j$. More realistic line-up selection would choose players directly proportional to their strength. 
We recall that the Elo-rating is translation invariant, hence we shift
the expected values $\theta$ to the interval $[0,10]$ in the
following.

As a first example consider the following football-inspired situation of $N=200$ teams with $M=23$ players each from which $m=11$ players are selected per match. 
For any team $T_i$ we then have $\theta_i=\frac{11}{23}\sum \rho_{i_k}$. We investigate two different initial setups for the teams:
\begin{enumerate}[label=($R$\arabic*)]
\item For every team $T_i$, $i=1, \ldots N$, the players' strengths
  $\rho_{i_k}$, $k=1,\dots 23$, is chosen randomly from the interval
  $\frac{1}{11}[5-\frac{5}{200}(i-1),5+\frac{5}{200}(i-1)]$. That is
  $\rho_{1_1}=\dots \rho_{1_{23}} =\frac{1}{11}5$ and $\rho_{200_k}\in
  \frac{1}{11}[0,10]$ for every $k$. In other words all teams have an
  approximate team strength of $\theta_i \approx 5$ with increasing variance $\sigma_i^2$ as can be seen in Figure~\ref{fig:mikro_var2}.
\label{enu:r1}
\item \label{enu:r2} For every team $T_i$, $i=1, \ldots N,$ the players' strengths $\rho_{i_k}$ are given by
\begin{align*}
\frac{1}{11}(4+\frac{6(i-1)}{197}+\eta_{i_k}),\quad k\in\{ 1,\dots ,23\},\quad \eta_{i_k}\in \mathcal{N}(0,1).
\end{align*}
The mean team strength of the first 198 teams is increasing 
from values around 4 to values around 10 and the variances $\sigma_i^2$ are of the same order. 
In addition we consider two teams, Germany ($i=199$) and Brazil
($i=200$), whose mean value and variance are motivated by the 2014
FIFA World Cup results, see \cite{i:GI}. We scale those values to $\theta_{\sf
  Ger}=10$ as well as $\theta_{\sf Bra}=9$. However, we do not scale the variance - it is signficantly higher in the generated data set, as can be seen in Figure~\ref{fig:mikro_var3}.
\begin{figure}[htb]
\centering
\begin{subfigure}[t]{.4\textwidth}
\includegraphics[height=4cm]{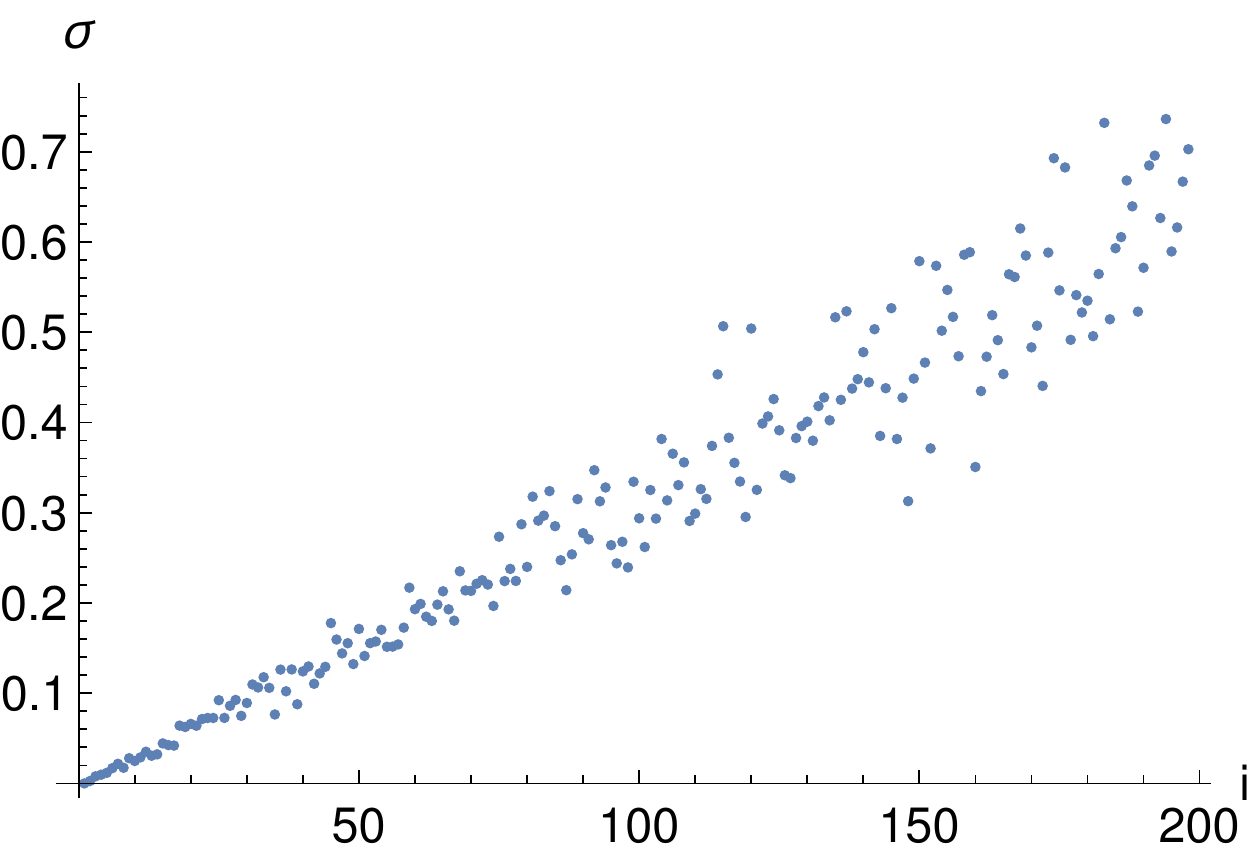} 
 \caption{The standard deviation $\sigma$ for every team $T_i$ based on rule \ref{enu:r1}.}\label{fig:mikro_var2}
  \end{subfigure} 
  \hspace*{1em}	
\begin{subfigure}[t]{.4\textwidth}
\includegraphics[height=4cm]{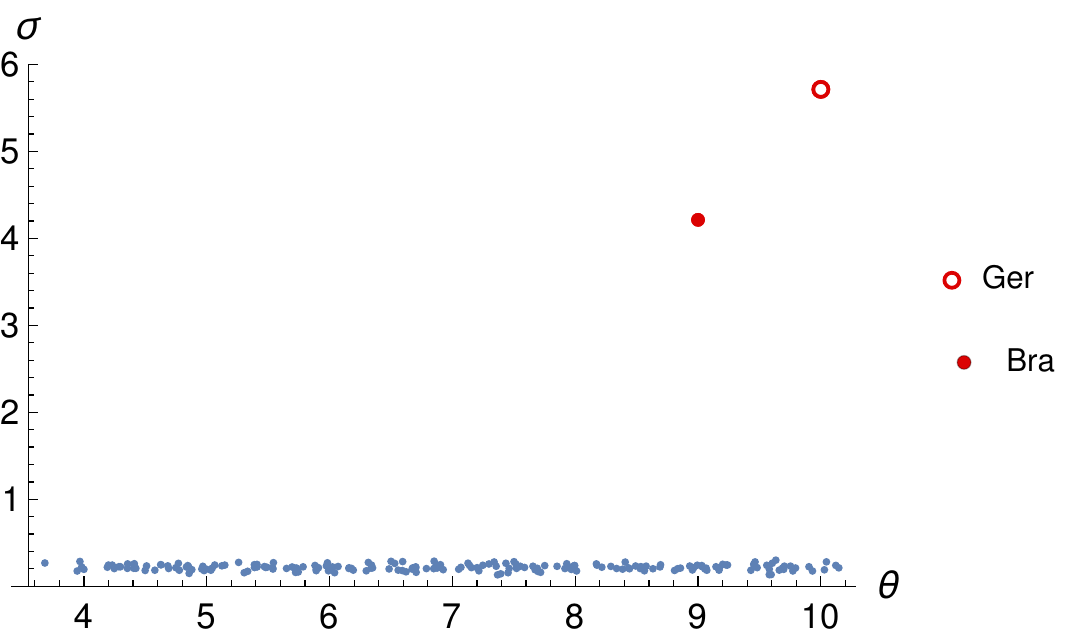} 
 \caption{The standard deviation $\sigma$ regarding $\theta$ based on rule \ref{enu:r2}.}\label{fig:mikro_var3}
  \end{subfigure}
 \caption{The standard deviation of the two setups visualised, both calculated over $10^4$ Monte Carlo experiments per team.}
 \label{fig:mikro_var}
\end{figure}
\end{enumerate}
We carry out direct Monte Carlo simulation using Bird's scheme, see
\cite{b:PLT13}, for these two initial setups. We choose time steps of $\Delta t=0.1 $
and perform $25$ matches per time step. 
The results were then averaged over 50 realisations after $2\cdot
10^6$ time-steps. Figure \ref{fig:mikro1} shows the final distribution
of teams for the two different initial setups. In
Figure~\ref{fig:mikro1a} we see clustering around the point $(\theta ,
R)=(5,5)$ as expected for setup \ref{enu:r1}. However, an interesting
phenomenon is that teams with $\theta <5$ consistently under-perform
and conversely $\theta >5$ over-perform, as they lie above and below
the line $\theta = R$, respectively. In Figure~\ref{fig:mikro1b}, we
see convergence towards a steady state for the 198 teams created using
the rule \ref{enu:r2}. However, this straight line has a steeper slope
than $\theta = R$. Furthermore, the German and Brazilian team are
clear outliers, both are under-performing relative to their strengths.
%
 
\begin{figure}[htb]
 \centering
\begin{subfigure}[t]{.4\textwidth}
 \hspace*{2em}
\centering
\includegraphics[height=4cm]{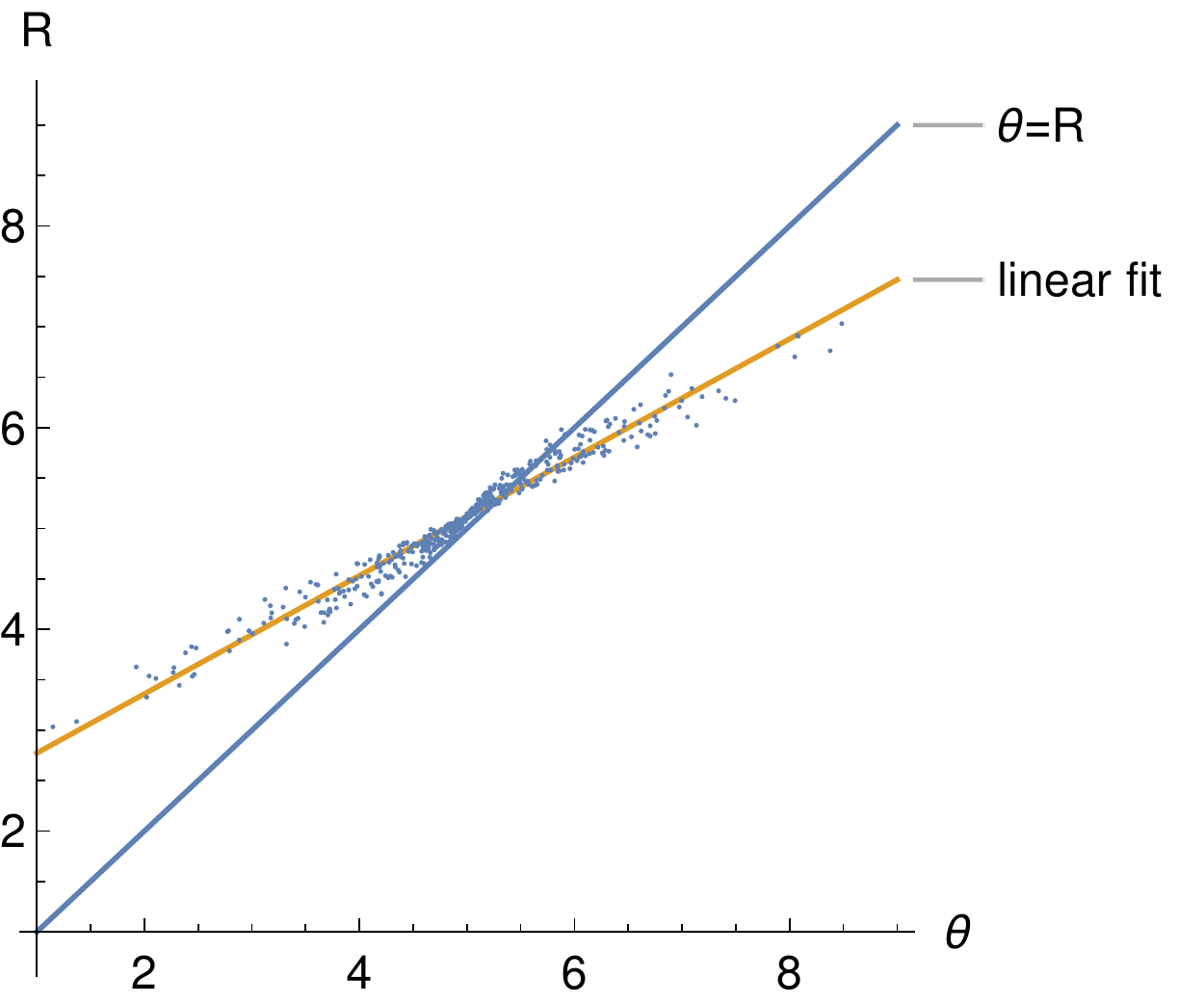} 
 \caption{Setup \ref{enu:r1} with increasing variance in the team performance.}\label{fig:mikro1a}
 \end{subfigure}
\hspace*{3em}
\begin{subfigure}[t]{.4\textwidth}
 \centering
\includegraphics[height=4cm]{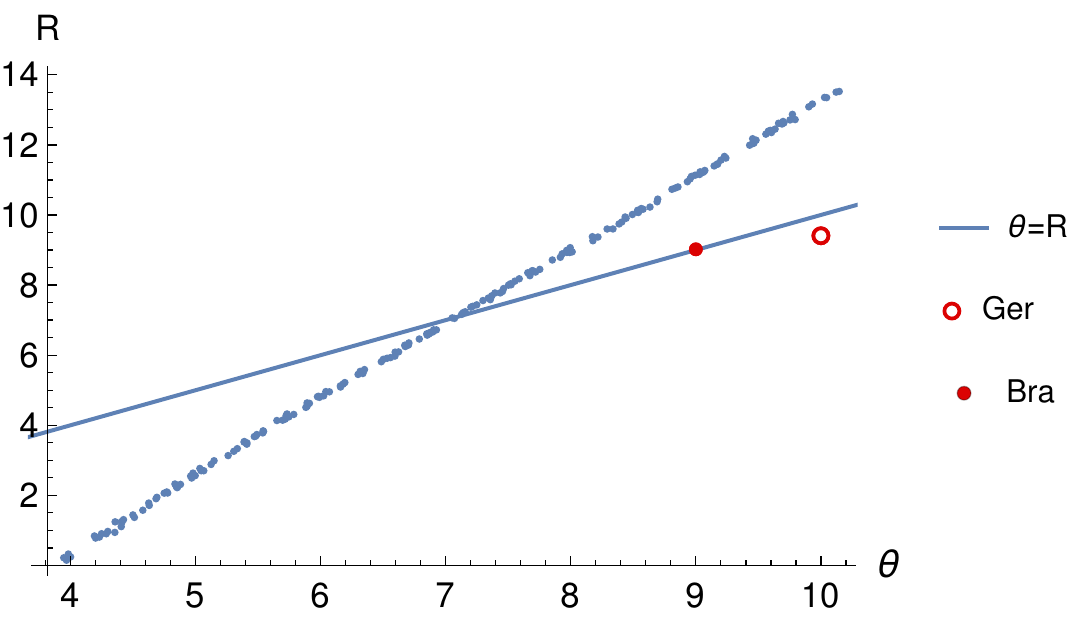} 
 \caption{Setup \ref{enu:r2} with constant variance $\sigma=1$ in the team performance.}\label{fig:mikro1b}
 \end{subfigure} 
 \caption{Stationary team distribution for setup \ref{enu:r1} and \ref{enu:r2}.}\label{fig:mikro1}
\end{figure}

\section{A macroscopic Elo-model for teams}\label{re:macroteams}

In general, the expected value $\theta$ and variance $\sigma^2$ of the
microscopic model \eqref{eq:eloorig} are finite since they
result from discrete, finite random processes. Compared to the formal
derivation of the macroscopic model in previous
works \cite{a:jabin, mt:K2016, b:elo}, we need additional assumptions on  the moments of $\sigma$ because of the unboundedness of $K$ in \eqref{eq:taylorE} in the 2$^{\text{nd}}$ and $3^{\text{rd}}$ argument when passing from micro to macro. 

Let $f(t,\theta,\sigma,r)$ be the distribution of teams at time $t$
with expected team performance $\theta$, variance $\sigma^2$ and
rating $r$. The derivation of the macroscopic model
\eqref{eq:kineticmodel} below is based on the following assumptions:
\begin{enumerate}[label=($A$\arabic*)]
 \item \label{a:finit} Let $f^0 \in H^1(\mathbb R^3)$ with $f^0 \geq 0$ and having compact support. Furthermore, we assume:
 \begin{align}
&\int_{\mathbb R^3} f^0(\theta,\sigma,r) d\theta d\sigma d r =1,\;\; \int_{\mathbb R^3} R f^0(\theta,\sigma,r)d\theta d\sigma d r = 0, \;\;
 \int_{\mathbb R^3} \theta f^0(\theta,\sigma,r) d\theta d\sigma d r = 0,\notag\\
&\int_{\mathbb R^3} \sigma f^0(\theta,\sigma,r) d\theta d\sigma d r = 1,\;\; \int_{\mathbb R^3}  \sigma^2 f^0(\theta,\sigma,r) d\theta d\sigma d r = C_{\sigma^2}.
 \end{align}
\item Let the interaction rate function $w\geq 0$ be an even function with $w \in C^2(\mathbb R^3) \cap L^{\infty}(\mathbb R^3)$.
\end{enumerate}
In Appendix \ref{re:appendix} we derive the following macroscopic
Fokker-Planck equation for the distribution of teams $f=f(t,r,\theta, \sigma)$:
\begin{align}\label{eq:kineticmodel}
\begin{split}
\frac{\partial }{\partial t} f(t,\theta,\sigma,r)+\frac{\partial }{\partial r}(a[f] f(t,\theta,\sigma,r))&=0, \qquad \text{ in }[0,T)\times \mathbb R^3, \\
f(t=0,\theta,\sigma, r)&=f^0(\theta,\sigma, r), \qquad \text{ in }\mathbb R^3 ,
\end{split}
\end{align}
with
\begin{align*}
a[f] & =\int\limits_{\mathbb{R}^3} w (r-r')
\big( b(\theta-\theta' )+ \frac{1}{2}b''(\theta-\theta')(\sigma^2+\sigma'^2)-b(r-r')\big)f(t,\theta',\sigma',r')\, d\theta' d\sigma' d r'.
\end{align*}
Here $b''$ is the second derivative introduced in \eqref{eq:taylorE}. We note that the operator $a[\cdot]$ includes an additional correction
term resulting from the variance of the distribution of strengths.
This term's sign depends on the sign of $\theta-\theta'$, either
decreasing or increasing the adjustment of ratings.
This is consistent with the under- and over-performance of teams
  observed in the microscopic simulations in Section~\ref{re:microteams}.\ref{subsec:nummericsmicro}.

\subsection{Analysis of the Fokker-Planck equation}\label{sec:analys}

Since $b$ and $b''$ are odd, the total mass is conserved as 
\begin{align*}
\frac{\partial }{\partial t} \int_{\mathbb{R}^3} f(t,\theta,\sigma,r)\,d\theta d\sigma d r =0.
\end{align*}
and therefore $\int_{\mathbb{R}^3}f(t,\theta,\sigma,r)\,d\theta d\sigma d r= \int_{\mathbb{R}^3}f^0(\theta,\sigma,r)\,d\theta d\sigma d r=1$ for all times $t>0$.

Next, we show the existence of a classical solution to
\eqref{eq:kineticmodel} following arguments from \cite{a:HT2008, a:BCP1997}. 
\begin{theorem}\label{theo:existence1}
Assume that the initial datum $f^0$ 
\begin{enumerate}[nosep]
\item is compactly supported in the phase space, i.e.\ $\operatorname{supp}_{(r,\theta,\sigma )}f^0$ is bounded,
\item is $C^1$-regular and bounded:
\begin{align*}
 \sum_{0 \leq |\alpha| \leq 1} \left\| \nabla_r^\alpha f^0 \right\|_{L^{\infty}} < \infty.
\end{align*}
\end{enumerate}
Then, for any $t \in (0, \infty)$, there exists a unique classical
solution $f \in C^1\big( [0, t) \times \RRR \big)$ to \eqref{eq:kineticmodel}.
\end{theorem}
\begin{proof}
In the following we consider the 'all play all' setting, that is $w
\equiv 1$; our arguments can, however, be generalised for 
interaction functions $w$ satisfying {\it(A2)}. We start by showing that the solution cannot blow up in finite time. Next we prove local in time existence based on a priori estimates and a fixed point argument. Global existence follows from a continuation argument using energy estimates.

First we show that the local solution $f$ remains uniformly bounded. We rewrite \eqref{eq:kineticmodel} in a non-conservative form, 
\begin{align}\label{eq:nonconservative}
\begin{split}
\frac{\partial }{\partial t} f(t,\theta,\sigma,r)+ a[f] \frac{\partial }{\partial r}f(t,\theta,\sigma,r)=- f(t,\theta,\sigma,r) \frac{\partial }{\partial r} a[f],
\end{split}
\end{align}
where we have
\begin{align*}
 -f(t,\theta,\sigma,r) \frac{\partial }{\partial r} a[f] = f(t,\theta,\sigma,r)\int\limits_{\mathbb{R}^3} b'(r-r') f(t,\theta',\sigma',r') \,d\theta' d\sigma' d r',
\end{align*}
which yields
\begin{align*}
\left\| \frac{\partial }{\partial r} a[f] \right\|_{L^{\infty}} \leq L,
\end{align*}
where $L$ is the Lipschitz constant of $b$ and we used that the total mass equals one.
Next we  consider a trajectory starting at time $\t_0\in\mathbb R^+$ in $(r_0,\theta_0,\sigma_0)$, then the characteristics are given by
\begin{align}
\frac{\partial }{\partial t} r(t)=a[f],\qquad \frac{\partial }{\partial t} \theta(t)=\frac{\partial }{\partial t} \sigma(t)=0.
\end{align}
Therefore,
\begin{align*}
\frac{\partial }{\partial t}f(t,\theta,\sigma,r) \leq L f(t,\theta,\sigma,r),
\end{align*}
and Gronwall's lemma gives
\begin{align}\label{eq:gronwall}
\left\| f(t) \right\|_{L^{\infty}} \leq e^{Lt}\left\|f^0 \right\|_{L^{\infty}}.
\end{align}
Hence, the solution cannot blow up in finite time. 

We continue with the existence of a local solution following Theorem 3.1 in \cite{a:BCP1997}. To this end we investigate the non-linear transport operator $\H$:
\begin{align}\label{eq:transport}
\H(f) = - f(t,\theta,\sigma,r) \frac{\partial }{\partial r}
  a[f],\qquad \text{with } \H:=\frac{\partial }{\partial t} + a[f] \frac{\partial }{\partial r},
\end{align}
in the following. There exist positive constants $C_1,C_2$  such that
\begin{align}\label{eq:bounds}
\begin{aligned}
& |{\H}(f)| \leq \left\|b' \right\|_{L^{\infty}} |f| =C_1 |f|,\\
& |{\H}(\frac{\partial }{\partial r}f)| \leq|f|\left\|b'' \right\|_{L^{\infty}} + \left| \frac{\partial }{\partial r}f\right|\left\|b' \right\|_{L^{\infty}}= C_2  \big( |f| + \left| \frac{\partial }{\partial r}f \right| \big) ,
\end{aligned}
\end{align}
because of $b$ being Lipschitz and therefore we have $L_\infty$ bounds for its derivatives. 
Moreover, the map $H$  defined by
\begin{align*}
H: C^1\big( [0, t) \times \RRR \big) \to C^1\big( [0, t) \times \RRR \big), \quad
f\mapsto (\H)^{-1} (-f \frac{\partial }{\partial r} a[f])
\end{align*}
is bounded since we can use the estimates \eqref{eq:bounds} together with the bounded inverse theorem. Next we consider the solution along the trajectories $(\theta (t),\sigma (t), r(t))$
\begin{align*}
f(t,\theta (t),\sigma (t), r(t))= f(0,\theta_0, \sigma_0, r_0)+\int_{0}^t  \frac{\partial }{\partial r}\bigg( a[f(t,\theta (t),\sigma (t), r(t))] f(t,\theta (t),\sigma (t), r(t))\bigg) dt.
\end{align*}
We can then use the previous estimates to choose a $t>0$ such that $H$
is a contraction. Using Banach's fixed point theorem we obtain a unique local solution $f\in C^1 \big( [0,t]\times \RRR \big)$.
Let $\F$ be the $W^{1, \infty}$-norm of $f(t)$,
\begin{align*}
\F (t)=\sum_{0 \leq |\alpha|\leq 1} \left\| \nabla_r^\alpha f(t) \right\|_{L^{\infty}}.
\end{align*}
Using \eqref{eq:bounds} and again Gronwall's lemma we get
\begin{align*}
\frac{\partial }{\partial t} \F (t)\leq C_3 \F (t)
\end{align*}
and therefore we have the upper bound
\begin{align*}
 \F (t) \leq \F (0)e^{C_3 t},\qquad \forall t\in [0,T).
\end{align*}
The energy bound in $W^{1,\infty}$ allows us to use  the standard continuation principle, giving the global extension of the local solution.
\end{proof}

We continue by analysing the behaviour of the moments of $f$.
We define the $s$-th moment for $s\in\mathbb{N}$ with respect to $r$
(and similar the moments with respect to $\theta,\sigma$),
\begin{align}\label{eq:momentfpe}
m_{s,r}(t)=\int_{\R^3}r^s f\,d\theta dr d\sigma.
\end{align}
The evolution with respect to $\sigma$ and $\theta$ is trivial, as the
function does not change with respect to these variables. The
evolution of the second moment w.r.t. to $r$ satisfies:
\begin{align}\label{eq:duality}
&\frac{d}{dt}\int\limits_{\RRR}r^2  f\,d\theta dr d\sigma
=-\int\limits_{\RRR}r^2\frac{d}{dr}(a[f]f)\,d\theta dr d\sigma \notag \\
&=2\int\limits_{\mathbb{R}^6}r w (r-r')
\Big( b(\theta -\theta')+ \frac{\sigma^2+\sigma'^2}{2} b''(\theta-\theta')-b(r-r')\Big) ff'\,d\theta'  dr' d\sigma' d\theta  dr d\sigma \notag \\
&=-\int\limits_{\mathbb{R}^6}r  b(r-r') w(r-r') f'f \,d\theta'  dr' d\sigma' d\theta  dr d\sigma
-\int\limits_{\mathbb{R}^6}r'  b(r'-r )w(r'-r) f'f\, d\theta'  dr' d\sigma' d\theta  dr d\sigma \notag \\
&\quad+2\int\limits_{\mathbb{R}^6}r  w(r-r')\big( b(\theta -\theta')+ \frac{\sigma^2+\sigma'^2}{2} b''(\theta-\theta')\big) ff'\,d\theta'  dr' d\sigma' d\theta  dr d\sigma \notag\\ 
&=-\int\limits_{\mathbb{R}^6}(r-r') b(r-r')  w(r-r')ff'\,d\theta'  dr'
                                                                                                                                                                                      d\sigma' d\theta  dr d\sigma \notag \\&<0.
\end{align}
Here, we used the short-hand notation $f' = f(t, \theta', \sigma', r')$. Furthermore, we used that for $r-r'<0$ the function $b(r-r')<0$ is negative, since $b$ is odd and monotonically increasing. The latter does not hold in general for $b+b''$, however, the second integral vanishes since the integrand is still odd in $\theta$ and $r$. Therefore, the second moment in $r$ decreases over time and we expect convergence towards a stationary state.
Our computational experiments confirm this expected convergence. However, we are not able to compute these stationary states explicitly as it was done in \cite{a:jabin}. 

\subsection{Numerical results for the macroscopic model}\label{subsec:nummericsmacro}

We perform several computational experiments illustrating the dynamics of \eqref{eq:kineticmodel} using a finite difference scheme. It is based on the generalisation of a finite difference scheme for conservation laws with discontinuous flux presented by Towers in \cite{a:99fvm}. This generalisation is straight-forward, as \eqref{eq:kineticmodel} has only transport in $r$ direction. Let $f^n_{j,l,m}$ denote the solution at the discrete points $(j\Delta r, l\Delta \theta, m\Delta \sigma)$, $j, l, m \in \mathbb{N},$ and time $t^n= n \Delta t$, $n \in \mathbb{N},$ with discrete positive increments $\Delta r, \Delta \theta$, $\Delta \sigma$ and $\Delta t$. Then the explicit scheme reads as follows:
\begin{align*}
f^{n+1}_{j,l,m}=f^n_{j,l,m}-\frac{\Delta t}{\Delta r}(a^n_{j+1/2,l,m} h_{j+1/2,l,m}^n-a^n_{j-1/2,l,m}h_{j-1/2,l,m}^n),
\end{align*}
with cell averages $a^n_{j+1/2,l,m}=\frac{1}{\Delta r}\int_{j}^{j+1}a[f^n(r,l\Delta\theta,m\Delta\sigma)]\,dr$. The function $h$ is chosen depending on the sign of the averaged flux (as in the usual Godunov scheme) that is
\begin{align*}
h^n_{j+1/2,l,m}=\begin{cases}
f^n_{j,l,m}, & a^n_{j+1/2,l,m}\geq 0,\\
f^n_{j+1,l,m}, & a^n_{j+1/2,l,m}< 0.
\end{cases}
\end{align*}
In Figure~\ref{fig:makro_1} we visualise the first marginals $\int_\R
f(t,r,\theta,\sigma)\,d\sigma$ and $\int_\R
f(t,r,\theta,\sigma)\,d\theta$ of the team distribution at time $t=5$
using timesteps of size $\Delta t=10^{-5}$ and $\nu=1$. The computational domain
is $\Omega=[0,10]\times [0,10]\times [0,1]$ and the spatial
discretisation was set to $\Delta r=\Delta \theta = \Delta
\sigma=5\cdot 10^{-2}$, the initial distribution of teams uniform and
normalised. The left  plot in Figure~\ref{fig:makro_1} shows that
ratings converge towards the mean strength for $\theta \in [6,8]$, but
are blurred for smaller and larger means. We observe a similar over- and
under-performance as in the microscopic simulations in Figure~\ref{fig:mikro1b}.
The right plot illustrates the decrease of $m_{2,r}$ in the direction $\theta$. The larger the uncertainty $\sigma$, the less accurate the ratings as all teams get a similar rating (around $7$).

\begin{figure}[htb]
\centering
\includegraphics[width=\textwidth]{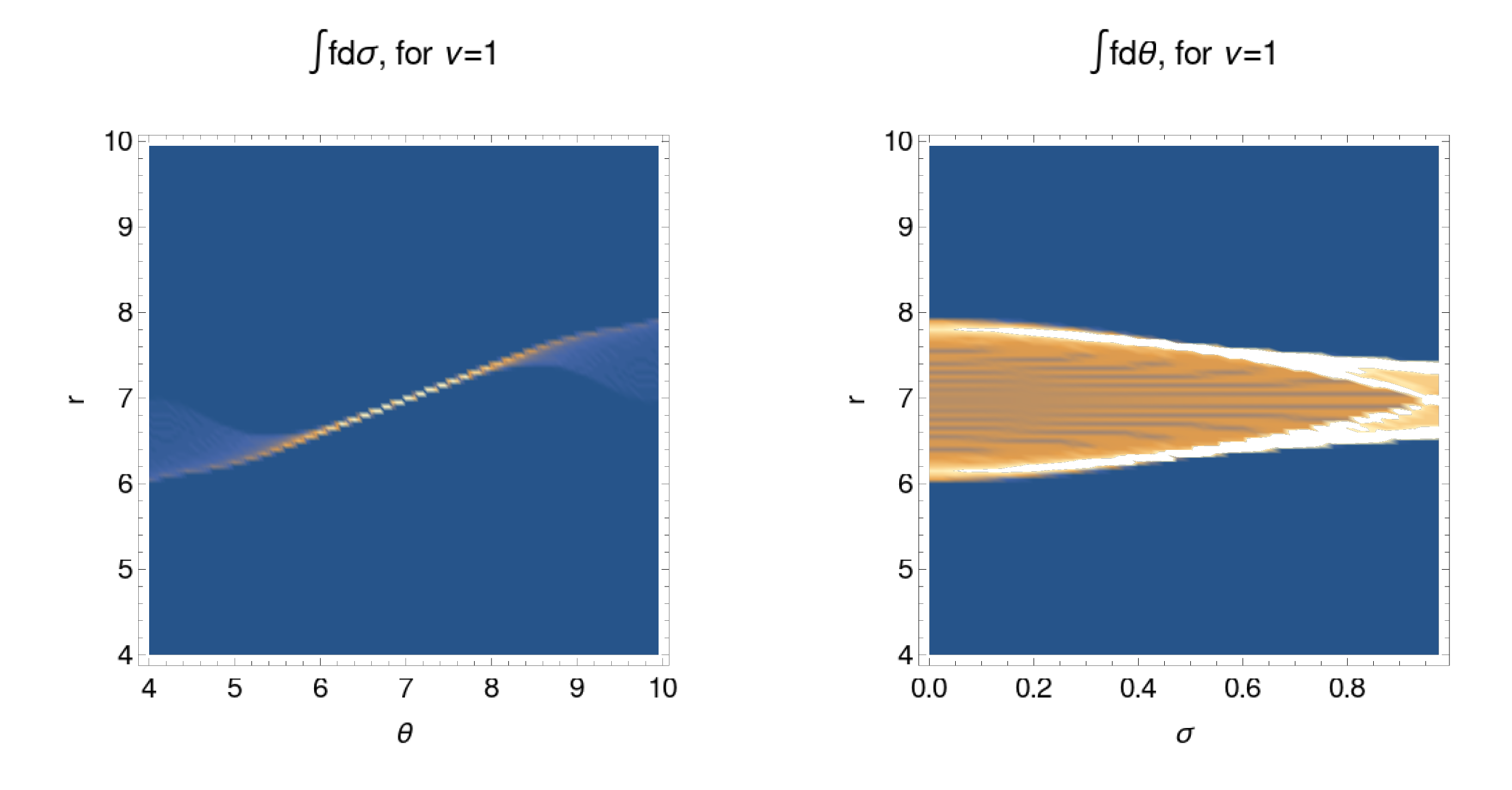} 
 \caption{Second and third marginal of the team distribution $f$ at time
   $t=5$ for $f^0=1$.}
\label{fig:makro_1}
\end{figure}



\section{Special scaling limits and homogeneous player distributions}
\label{s:scalinglimits}

Our microscopic computational results suggest that if all teams have
the same player variance, then the ratings converge to the underlying
mean team strength. In this case, however, the integral over $\sigma$
can be seen as a point evaluation and we can simplify
\eqref{eq:kineticmodel} for constant $\sigma\in\R^+_0$:
\begin{align}\label{eq:kineticmodel2}
\begin{split}
\frac{\partial }{\partial t} f (t,\theta,r)+\frac{\partial }{\partial r}(a[f ] f (t,\theta,r))&=0\\
f (t=0,\theta, r)&=f^0(\theta, r),
\end{split}
\end{align}
with $a$ changed to
\begin{align}\label{eq:kernel2}
a[f] & =\int\limits_{\RR} w (r-r')
\big( b(\theta-\theta' )+ \sigma^2 b''(\theta-\theta')-b(r-r')\big)f (t,\theta',r') dr'  d \theta'.
\end{align}

We discuss the existence of a unique solution and the analysis of the
moments. Furthermore we consider the relative energy to prove convergence of the
team strengths to ratings.
The existence of a classical solution itself follows from the same arguments as in Theorem~\ref{theo:existence1}. 
\begin{theorem}\label{theo:existence2}
Assume that the initial datum $f^0$
\begin{enumerate}[nosep]
\item is compactly supported in the phase space, i.e.\ $\su_{(r,\theta )}f^0$ is bounded
\item is $C^1$-regular and bounded:
\begin{align*}
\displaystyle \sum_{0 \leq |\alpha| \leq 1} \left\| \nabla_r^\alpha f^0 \right\|_{L^{\infty}} < \infty.
\end{align*}
\end{enumerate}
Then, for any $T \in (0, \infty)$, there exists a unique classical
solution $f \in C^1\big( [0, T) \times \RR)$ to \eqref{eq:kineticmodel2}.
\end{theorem}
\noindent The proof of Theorem \ref{theo:existence2} can be easily
adapted from the proof of Theorem \ref{theo:existence1} and is omitted here. Equation \eqref{eq:kineticmodel2} is conservative, hence the total mass is preserved, and the moments with respect to $\theta$ is zero. Again, the second moment w.r.t. to $r$ is decreasing (using similar arguments as in \eqref{eq:duality}):
\begin{align*}
\frac{d}{dt}\int\limits_{\RR}r^2  f\,d\theta dr
&=-\int\limits_{\RR}r^2\frac{d}{dr}(a[f]f)\,d\theta  dr d\theta'  dr'\\
&=-\int\limits_{\mathbb{R}^4}(r -r') \big( b(\theta -\theta')+ \sigma^2 b''(\theta-\theta')\big) w(r-r') ff'\,d\theta'  dr'd\theta  dr<0,
\end{align*}
using the short hand notation $f' = f(t,\theta', r')$.\\

\noindent The ratio of $\nu$ and $\sigma$ is important in order to be able to show the convergence of the team ratings to the average strength. This is the case under following assumption:  
\begin{enumerate}[label=($B'$)]
\item $b+\sigma^2 b''$ is monotonically increasing. \label{as:B2}
\end{enumerate}
Note that \ref{as:B2} holds for example for $b(z)=\tanh (\nu z)$  if $1 +  \nu^2 \sigma^2 (4 -  6 \operatorname{sech} (z \nu)^2) >0$. Then we can use similar arguments as Jabin and Junca \cite{a:jabin}, who considered the relative energy
\begin{align*}
\E (t)=\int\limits_{\mathbb{R}^2}(r-\theta)^2  f (t,\theta, r) \,dr d\theta.
\end{align*}
In the following we will show that
\begin{align}\label{eq:energydecay}
\frac{d \E(t)}{dt}< 0.
\end{align}
We calculate:
\begin{align*}
\frac{d}{dt}\int\limits_{\RR}(r-\theta)^2  f\,d\theta dr
&=-\int\limits_{\RR}(r-\theta)^2\frac{d}{dr}(a[f]f)\,d\theta  dr d\theta'  dr'\\
&=-\int\limits_{\mathbb{R}^4}(r-r') b(r-r')  w(r-r')ff'\,d\theta'  dr'd\theta  dr\\
&\quad-\int\limits_{\mathbb{R}^4}(\theta -\theta') \big( b(\theta
                                                                                      -\theta')+ \sigma^2 b''(\theta-\theta')\big) w(r-r') ff'\, d\theta'  dr'd\theta  dr <0.
\end{align*}
For $r-r'<0$ we have $b(r-r')<0$, while the opposite holds true for $(r-r')>0$. Because of \ref{as:B2} the second term is positive, yielding the stated energy decay. 

Assumption \ref{as:B}  together with  \ref{as:B2} gives us bounds for $b'''$, whereas we can deduce $b''$ being Lipschitz, too, with constant $L_2$. Following the arguments in Jabin and Junca, \cite{a:jabin}, we obtain:
\begin{theorem}
Let $f^0$ be as in Theorem \ref{theo:existence2} and $w\ge w_{min}>0$ on $\su f^0$. Then
\begin{align*}
\E (t) \leq \E(0)\exp(-2 w_{min}(L+\sigma^2 L_2)t),
\end{align*}
where $L,L_2$ depend on $b,b''$ (and therefore $\nu$) and on $\su
f^0.$ 
\end{theorem}
\begin{proof}
The proof is along the lines of \cite{a:jabin, mt:K2016}, adapted for
the additional term related to $b''$.  
We define
\begin{align*}
\mathcal D(f(t)):=\int\limits_{\mathbb{R}^4}(r-r'-\theta+\theta')w(r-r')[b(\theta-\theta')+\sigma^2 b''(\theta-\theta')-b(r-r')] ff'\, dr' d\theta' dr d\theta, 
\end{align*}
and by our previous calculations we have that $-\mathcal
D(f(t))=\frac{d \mathcal E(t)}{dt}\leqslant 0$.
Using similar symmetry arguments as before we deduce that 
\begin{align*}
-\int\limits_{\mathbb{R}^4} (r-r')w_{min}\sigma^2 b''(r-r')ff'\,d\theta'  dr'd\theta  dr<0.
\end{align*}
We now  split the integrands of $\mathcal D$ to obtain
\begin{align*}
\mathcal D(f(t)) \geq &
                        \int\limits_{\mathbb{R}^4}(r-r'-\theta+\theta')w_{min}[b(\theta-\theta')-b(r-r')]
                        ff'\, dr' d\theta' dr d\theta \\
& + \int\limits_{\mathbb{R}^4}(r-r'-\theta+\theta') w_{min} \sigma^2 [
                                                           b''(\theta-\theta')-b''(r-r')]ff'\,
                                                           dr'
                                                           d\theta'
                                                           dr d\theta.
\end{align*}
Using that $b,b''$ are Lipschitz and odd we have
\begin{align*}
(r-r'-\theta+\theta') w_{min}[b(\theta-\theta')-b(r-r')] & \geq L w_{min} \vert r-r'-\theta+\theta'\vert ^2 \\
(r-r'-\theta+\theta') w_{min}\sigma^2[b''(\theta-\theta')-b''(r-r')] & \geq \sigma^2 L_2 w_{min} \vert r-r'-\theta+\theta'\vert ^2.
\end{align*}
Therefore, it follows that
\begin{align*}
\mathcal D(f)\geqslant\int\limits_{\mathbb{R}^4}(L+\sigma^2 L_2) w_{min} \vert r-r'-\theta+\theta'\vert ^2  f' f\, dr' d\theta' dr d\theta.
\end{align*}
We assume w.l.o.g. (due to the translation invariance of the model)
\begin{align}
\int\limits_{\mathbb{R}^2}r  f(t,r,\theta)\ dr\ d\theta =\int\limits_{\mathbb{R}^2}\theta  f(t,r,\theta)\, dr d\theta =0
\end{align}
which gives 
\begin{align*}
\int\limits_{\mathbb{R}^4}(r-\theta)(r'-\theta')\ f' f\, dr' dr d\theta' d\theta=0. 
\end{align*}
Then we can deduce
\begin{align*}
\mathcal D(f)\geqslant 2 (L+\sigma^2 L_2) w_{min}\int\limits_{\mathbb{R}^2}\vert r-\theta\vert^2\ fdr\ d\theta 
\end{align*}
and altogether 
\begin{align*}
-\mathcal D(f)=\frac{d \mathcal E(t)}{dt}\leqslant -2 L w_{min}\int\limits_{\mathbb{R}^2}\vert r-\theta\vert^2\ f\ dr\ d\theta= -2 w_{min} (L+\sigma^2 L_2) \mathcal E(t).
\end{align*}
Using Gronwall's lemma we conclude the proof as in \cite{a:jabin, mt:K2016}.
\end{proof}




\noindent We conclude by underpinning our analytical results with numerical simulations.

\paragraph{Micro- and macroscopic simulations.}\label{subsubsec:num2}
For the microscopic simulation we consider $N=500$ players with fixed
mean strengths
$\theta_n$, chosen uniformly distributed in $[4,10]$. In every time-step we then choose $\mathcal{N}(\theta_n,\sigma)$ distributed values for the evaluation of $\Sij$. We set $\nu=0.5$ and simulate  $10^6$ time-steps of $\Delta t=0.1$ and 25 collisions per time-step over 50 realisations as in the previous simulations described in Section~\ref{re:microteams}\ref{subsec:nummericsmicro}. 
On a macroscopic level, we use the algorithm presented in
Section ~\ref{re:macroteams}\ref{subsec:nummericsmacro} reduced by
the dimension in $\sigma$.

We see a great agreement between the two
models in Figure \ref{fig:mikro_a11} and \ref{fig:makro_a11}. In
addition, we clearly see the influence of $\sigma$ on the ranking as
discussed in Figure~\ref{fig:makro_1}. If the variance $\sigma$ is
large, all teams are rated equally, in particular the ratings converge
to $7$ for all values of $\theta$.  Or expressed differently: in
expectation weaker teams are over-performing and stronger teams under-performing. We see a similar effect already in our first microscopic simulations in Figure~\ref{fig:mikro1}. 

\noindent Moreover, the numerical simulations show that $\nu$ can be
used to balance the variance $\sigma^2$ and obtain the desired
convergence of ratings to the teams' average strength as discussed in
the previous subsection. If assumption \ref{as:B2} holds the long-term
behaviour of \eqref{eq:kineticmodel2} will be similar to the original
Elo model \eqref{eq:origelo} in \cite{a:jabin}. This effect can also
be observed on a microscopic level, see Figure~\ref{fig:mikro_nu},
where we compare the long-term behaviour for $\sigma=2$ and different values of $\nu\in \{1,0.1,0.01\}$. For $\nu=1$ all expected values $\theta\in [4,10]$ converge to $7$ and we get an almost horizontal line in the long-term run. If we choose $\nu=0.1$ we get the desired diagonal $\theta=r$ as time increases, as in the case $\nu=0.01$. However, the smaller value of $\nu$ corresponds to a slower convergence towards the stationary state, as can be seen in Figure~\ref{fig:mikro_nu}, bottom right. 

\begin{figure}[htb]
\centering
\includegraphics[width=\textwidth]{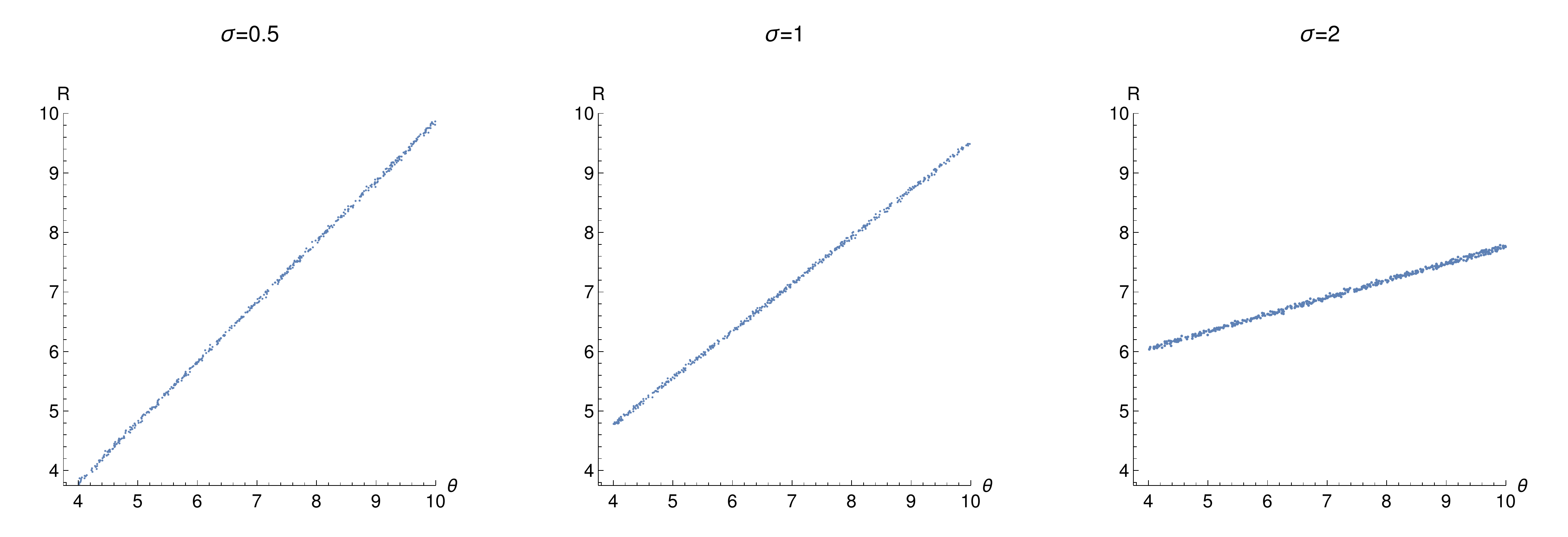} 
 \caption{Microscopic team distributions for $\nu=0.5$ and different values of $\sigma$.}
\label{fig:mikro_a11}
\end{figure}

\begin{figure}[htb]
 \centering
\includegraphics[width=\textwidth]{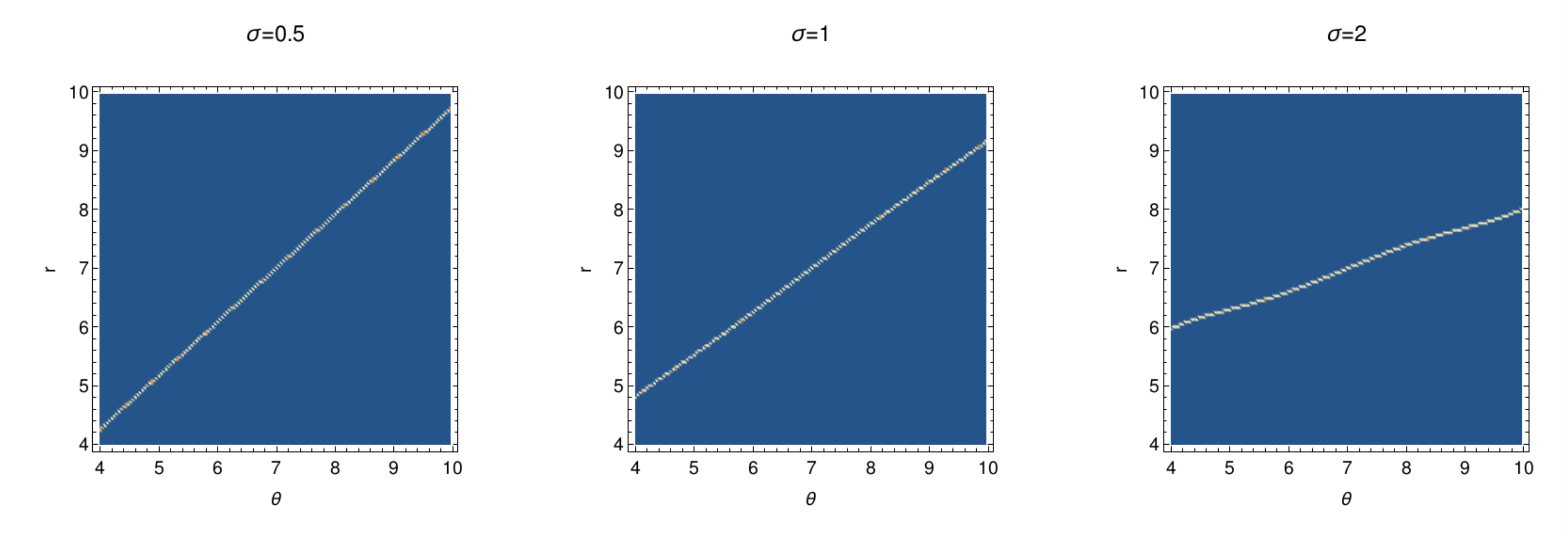} 
 \caption{Macroscopic results for  $\Omega=[4,10]\times [4,10]$ for
   different values of $\sigma$ and $\nu=0.5$.
 }\label{fig:makro_a11}
\end{figure}

\begin{figure}[htb]
 \centering
\includegraphics[width=\textwidth]{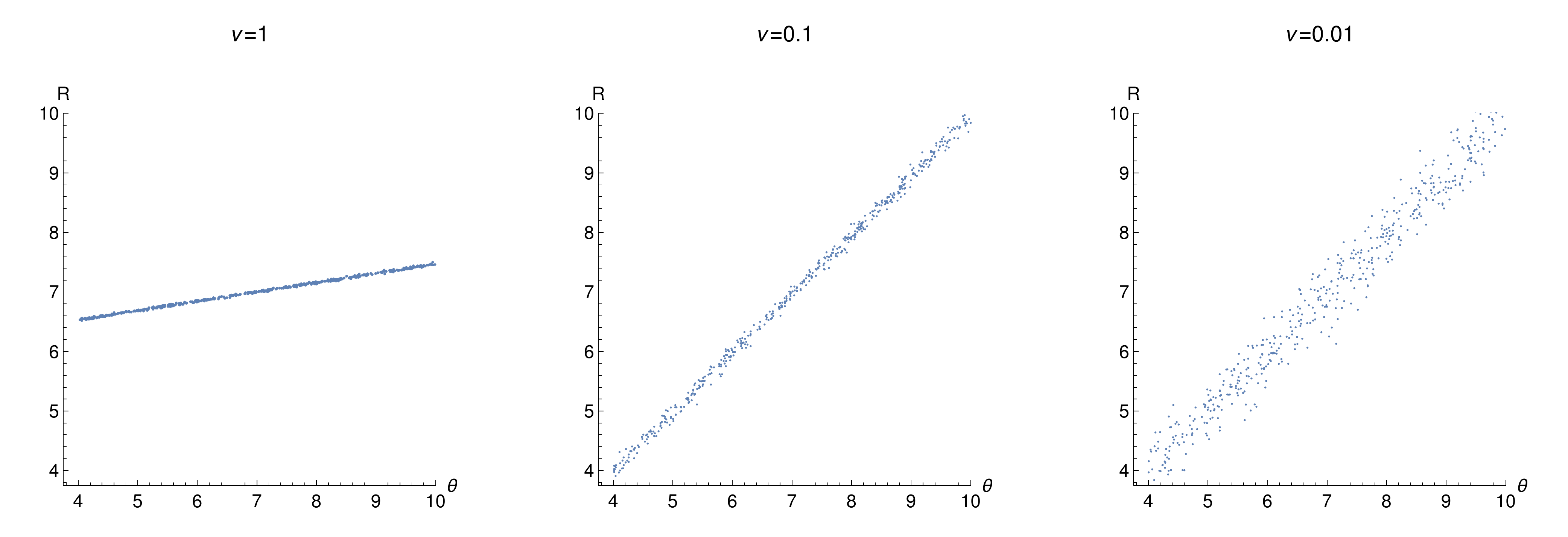}
 \caption{A smaller value of $\nu$ leads to slower convergence, however, we can retrieve the desired convergence of $R$ to the expected value $\theta$ by choosing $\nu$ sufficiently small. In all simulations we set $\sigma=2$.}\label{fig:mikro_nu}
\end{figure}

\section{Conclusion}\label{re:conclusion}

In this paper we proposed a generalisation of the Elo-rating model for
teams of players with varying strengths, which includes fluctuations
in the performance to account for example for variable line-ups in
team sports. Based on the microscopic interaction rules we then
derived the corresponding kinetic model, proved existence of a
solution and analysed different moments of its solution. These
analytical insights indicate the formation of non-trivial steady
states -- a hypothesis that is supported by our numerical results. 
Furthermore, we considered the special case of similar variance
$\sigma^2$, which allowed to formally derive a lower dimensional
equation. Under further smallness assumptions we could then use
techniques from Junca and Jabin, see \cite{a:jabin}, to show
convergence of the rating $r$ to the expected value $\theta$. The
smallness assumption relates to practically relevant parameter values. For example, in chess the scaling parameter $\nu$ in $b$ is usually quite small, around $\frac{1}{400}$, as reported in \cite{b:elo, b:fide, b:harkness}.  We were able to show numerically, both at the microscopic and kinetic level, that a large $\nu$ leads to the loss of convergence of $r\to\theta$. Choosing $\nu$ according to \ref{as:B2}, we obtain the desired convergence and were able to proof this analytically. This effect also occurs at the microscopic level.

 Nevertheless, the microscopic simulations showed that a large
$\sigma$ has a strong impact on the ratings. The
question therefore remains whether fluctuations in the underlying strength
should be included in $\rho$, see \cite{bib:During:elo},
or incorporated in the outcome of the game $\Sij$ (as proposed in this
paper). Following  \cite{bib:During:elo},
performance fluctuations could also be included via an additional
random term in the microscopic interactions. This leads to a PDE with
a diffusive term which is of
the following form:
\begin{align*}
\frac{\partial f(r,\theta, t)}{\partial t}= - \der \left(a[f]f(\theta,R,
  t)\right)+\frac{\s}{2} d[f]  \frac{\partial^2 }{\partial \theta^2}
f(\theta,r, t),
\end{align*}
with
\begin{align*} 
&a[f]=a[f](r,\theta, t)=\int_{\R^2}w(r-r') (b(\theta-\theta')-b(r-r'))f(\theta',r',t)\,d\theta' d r',\\
&d[f]=d[f] (r,\theta, t)= \int_{\R^2}w(r-r') f( r',\theta' ,t)\,d\theta' d r'
\end{align*}
where the influence of diffusion is determined by the maximum variance $\sigma^2$ of the team strengths. 

Accounting for uncertainty in ratings through an additional functional
dependence and not via diffusion also happens in the Glicko rating \cite{a:glicko98}, an
extension of the Elo rating. However, here the variable $\sigma$ is
the uncertainty of the rating. It is assumed that $\sigma$ increases
if players do not compete and decreases if they participate in
tournaments. This microscopic model could similarly be used to derive
a continuous kinetic rating model. Another interesting direction of future research is the combination of performance fluctuations with learning effects, as considered in \cite{bib:During:elo, mt:K2016}.

\subsection*{Contributions and Acknowledgments}
The paper has been conceived and is based on work from all three
  authors. MF obtained the analytical results under guidance from BD
  and MTW. MF carried out the numerical experiments with support from
  MTW. All authors worked on and approved the manuscript.

MF acknowledges partial support via the Austrian Science Fund (FWF) project F65. MTW acknowledges partial support via the New Frontier's Project NST-0001 of the Austrian Academy of Sciences \"OAW.

\appendix\section{Appendix}\label{re:appendix}
\subsection{Derivation of the Boltzmann-type equation}\label{re:teamsderivation}
We follow the derivation of \cite{bib:During:elo} for the corresponding PDE of Fokker-Planck type to study the dynamics of the corresponding model. We start with the evolution equation for the distribution of teams $\fe = \fe(\theta ,\sigma, R, t) $ with respect to their rating $R$, intrinsic team-strength $\rho$ and the variance $\si$.
For a fixed number of teams,  $N$, the interactions \eqref{eq:eloorig} induce a discrete-time Markov process  with
$N$-particle joint probability distribution
$P_N(\theta_1,\sigma_1,R_1,\theta_2,R_2,\dots,\theta_N,\sigma_N, R_N,\tau)$.
Then we can state the evolution of the first marginal from
\begin{align*}
P_1(\theta,\sigma, R,\tau)=\int P_N(\theta,\sigma,  R,\theta_2,\sigma_2, R_2,\dots, \theta_N ,\sigma_N, R_N,\tau) 
d\theta_2 d\sigma_2  dR_2 \cdots  d\theta_N d\sigma_N dR_N,
\end{align*}
where $\tau$ is the discrete time step
using only the one- and two-particle distribution functions \cite{b:cerc,a:cerc} in a single time step,
\begin{align*}
\begin{split}
& P_1(\theta,\sigma,  R,\tau+1)-P_1(\theta,\sigma,  R,\tau)=\\
& \Bigg\langle \frac 1N \Biggl[\int_{\mathbb{R}^6}
P_2(\ti,\si, R_i, \tj,\sj,  R_j,\tau) w(\Ri-\Rj)\bigl( \delta_0(\theta-\ti^*, R-\Rit)+\delta_0(\rho-\rjt,R-\Rjt) \bigr)\cdot \\ 
& \cdot d\ti d\si dR_i d\tj d\sj dR_j -2P_1(\theta,\sigma, R,\tau) \Biggr ]\Biggr\rangle.
\end{split}
\end{align*}
Here, $\langle\cdot\rangle$ denotes the mean operator with respect to
the random variables $\Sij$ and the function $w(\cdot)$ corresponds to the interaction rate function which depends on the difference of the ratings. 
This yields a hierarchy of equations, the so-called
BBGKY-hierarchy, see \cite{b:cerc,a:cerc}, describing the dynamics of the system of a
large number of interacting agents.

A standard approximation is to neglect correlations and assume that
\begin{align*}
P_2(\ti,\si, R_i, \tj,\sj,  R_j,\tau)=P_1(\ti,\si,  R_i,\tau)P_1(\tj,\sj,  R_j,\tau),
\end{align*}
By scaling time as $t=2\tau/N$ and performing the thermodynamical
limit $N\to\infty$, we can use standard methods of kinetic theory
\cite{b:cerc,a:cerc} to show that the time-evolution of the
one-agent distribution function $\fe$ (corresponding to $P_1$ and $\fe \fe$ to $P_2$) is governed by  the
following  Boltzmann-type equation:
\begin{align}\label{eq:bolt2}
\begin{aligned}
& \frac{d}{dt}\int_{\D}   \phi(\theta,\sigma,r )\fe(\theta,\sigma,r ,t) d\theta d\sigma dr =  \\ 
& \frac{1}{2} \Bigg\langle \int_{\D}\int_{\D}    \Big(\phi(\theta,\sigma,r^*)+\phi(\theta',\sigma',r'^*)-\phi(\theta,\sigma,r )-\phi(\theta',\sigma',r' ) \Big)\cdot \\
& \cdot w(r-r')\fe(\theta,\sigma,r ,t)\fe(\theta',\sigma',r' ,t) d\theta' d\sigma' dr' d\theta d\sigma dr  \Bigg\rangle,
\end{aligned}
\end{align}
where $\phi(\cdot)$ is a (smooth) test function, with support
$\mathrm{supp}({\phi})\subseteq \D$. 

\subsection{Analysis of the Boltzmann-type equation}
\subsubsection*{Conservation of mass}
Setting  $\phi(\theta,\sigma,r )=1$ in equation \eqref{eq:bolt2} we have
\begin{align*}
\frac{d}{dt} \int_{\R^3} \fe(\theta,\sigma,r ,t) d\theta d\sigma dr=0.
\end{align*}
Therefore, the total mass is conserved, that is
\begin{align*}
\int_{\R^3} \fe(\theta,\sigma,r ,t) d\theta d\sigma dr=1, \ \text{ for all times } t \geq 0.
\end{align*}

\subsubsection*{Moments with respect to the rating}
We define the $s$-th moment for $s\in\mathbb{N}$ with respect to $r$
\begin{align}\label{eq:moment}
m_{s,r}(t)=\int_{\R^3}r^s \fe(\theta,\sigma,r ,t) d\theta d\sigma dr.
\end{align}
Now choose $\phi(\theta,\sigma,r )=r$. Due to \ref{as:B}, \ref{a:finit} and the symmetry of $b(\cdot)$ and $b''(\cdot)$ we obtain 
\begin{align}\label{eq:conm1}
\begin{aligned}
\frac{d}{dt} m_{1,r}(t)=&\frac{1}{2}\var\int_{\R^6}  \fe(\theta,\sigma,r ,t)  \fe(\theta',\sigma',r' ,t) w(r-r')\cdot \\
&  \cdot \bigg( b(\theta-\theta')-\frac{1}{2}b''(\theta-\theta')(\s+\sigma'^2)-b(r-r')+ \\
& + b(\theta'-\theta)+\frac{1}{2}b''(\theta'-\theta)(\s+\sigma'^2)-b(r'-r)\bigg) d\theta' d\sigma' dr' d\theta d\sigma dr=0.
\end{aligned}
\end{align}
Hence the mean value w.r.t. the rating is preserved in time and therefore $m_{1,r}=0$ for all times $t\geq 0$. 


\subsubsection*{Moments with respect to the variance and expected value}
We define $m_{s,\theta}$ and $m_{s,\s}$ similar to \eqref{eq:moment} and also need the boundedness of the second moment $m_{2,\s}$.
\begin{align}\label{eq:conm2s}
\frac{d}{dt} m_{2,\s}(t)=0
\end{align}
which follows directly from \eqref{eq:bolt2} when testing with  $\phi(\theta,\sigma,r )= \s$. Analogue we get
\begin{align}\label{eq:conm2theta}
\frac{d}{dt} m_{1,\theta}(t)=\frac{d}{dt} m_{2,\theta}(t)=0.
\end{align}

\subsection{Derivation of the Fokker-Planck equation}
We now derive the limiting Fokker-Planck
equation in the case $\gamma \rightarrow 0$.
Based on the interaction rules \eqref{eq:eloorig}, which define the outcome of a game, we compute the expected values of the following quantities:
\begin{align}\label{eq:erwartungswerte}
& \langle r^*-r\rangle  =\var (\langle 
S
 \rangle -b(r-r'))\notag\\
& \Var[r^*-r]  = \var ^2 \Var [\Sij]\\
& \langle(r^*-r)^2\rangle  =\var ^2(\langle 
S
 \rangle-b(r-r'))^2+\Var[r^*-r]= \var ^2\big(\langle 
S
 \rangle-b(r-r'))^2+\Var [
S
]\big)\notag
\end{align}
with $S$ analogue to \eqref{eq:taylorE} where we used $\langle X^2 \rangle =\langle X \rangle^2 +\operatorname{Cov}[X,X]=\langle X \rangle^2+\Var [X]$. 
Using Taylor expansion of $\phi(\theta,\sigma,r^*)$ up to order two around $(\theta,\sigma,r )$, we obtain
\begin{align*}
\langle \phi(\theta,\sigma,r^*)-\phi(\theta,\sigma,r ) \rangle = & \langle r^*-r \rangle\der \phi(\theta,\sigma,r )+\frac{1}{2}\langle(r^*-r)^2\rangle\derr \phi(\theta,\sigma,r )+ \\
&+\langle \Rem (\phi, \theta,\sigma,r ,\t)\rangle, 
\end{align*}
where the remainder term  $\Rem$ is given in the Peano-representation of Taylor's formula via
\begin{align*}
\langle\Rem (\phi, \theta,\sigma,r ,\t) \rangle & = \frac{1}{2}\langle(r^*-r)^2\rangle \derr (\phi (\theta,\sigma, \bar{r})-\phi (\theta,\sigma, r))\\
& = \frac{1}{2}\var ^2\big(\langle 
S
 \rangle-b(r-r'))^2+\Var [
 S
  ]\big)\derr (\phi (\theta,\sigma, \bar{r})-\phi (\theta,\sigma, r))
\end{align*}
for  some $0\leq c \leq 1$ with $\bar{r}$  defined as 
\begin{align*}
\bar{r}= c r+(1-c)r^*.
\end{align*}
\noindent Next we rescale time as $\t=\var t$ and insert the expansion in \eqref{eq:bolt2}. This yields 
\begin{align*}
\begin{split}
 \frac{d}{d\t}\int_{\R^3} & \phi(\theta,\sigma,r )\fe(\theta,\sigma,r ,\t) d\theta d\sigma dr 
 = \frac{1}{2\var}\int_{\R^3} \Remm (\phi,\theta,\sigma,r ,\t)\fe(\theta,\sigma,r ,\t)d\theta d\sigma dr + \\
+\int_{\R^6} & \bigg( \der \phi(\theta,\sigma,r )\big( b(\theta-\theta')+ K(\theta-\theta',\sigma,\sigma')-b(r-r')\big) \cdot \\
 & \cdot w(r-r')\fe(\theta,\sigma,r ,\t) \fe(\theta',\sigma',r' ,\t) \bigg) d\theta' d\sigma' d r' d\theta d\sigma r 
\end{split}
\end{align*}
whereas the remainder $\Remm$ is given by
\begin{align}
\begin{split}
\Remm (\phi,\theta,\sigma,r ,\t)=   \int_{\R^3} \langle \Rem &(\phi, r'^*,\theta',\sigma',r' ,\t)\rangle w(r'-r)\fe(\theta',\sigma',r' ,\t)d\theta' d\sigma' dr' \\
 + \var^2\int_{\R^3} & \dersrs \phi (\theta',\sigma',r' ) \bigg( \big(\langle 
S 
 \rangle-b(r'-r)\big)^2+ \Var [
 S
 ] \bigg) \cdot \\
&  \cdot  w(r-r')\fe(\theta',\sigma',r' ,\t)d\theta' d\sigma' dr'.
\end{split}
\end{align}

All summands will vanish for $\gamma\to 0$ with similar arguments as in \cite{bib:During:elo}. Let us assume that $\phi(\theta,\sigma,r )$ belongs to the space $\mathcal{C}_{2+\delta}(\R^3)=\{h:\R^3\rightarrow\R, \ \|D^\zeta h\|_\delta<+\infty\}$, where $0<\delta\leq 1$, $\zeta$ is a multi-index with $|\zeta|\leq2$ and the seminorm $\|\cdot\|_\delta$ is the usual H\"older seminorm
\begin{align*}
 \|f\|_\delta=\sup_{x,y\in \R^3}\frac{|f(x)-f(y)|}{|x-y|^\delta}.
\end{align*}
Equations \eqref{eq:taylorE},\eqref{eq:taylorV} together with conservation laws \eqref{eq:conm2s} and \eqref{eq:conm2theta} guarantee the boundedness of both expectation $\langle 
S
 \rangle$ and variance $ \Var [
S
 ]$. Then with this choice of $\phi(\theta,\sigma,r )$, both summands  containing $\derr \phi$ vanish using the same arguments as in \cite{a:06toscani,a:09pareschi}.

\noindent Therefore, the density $\fe(\theta,\sigma,r ,\t)$ converges to $f(\theta,\sigma,r ,\t)$ 
which solves
\begin{align} \label{eq:weakeq}
\begin{aligned}
\frac{d}{d\tau}  \int_{\R^3} & \phi(\theta,\sigma,r)f(\theta,\sigma,r ,\t) d\theta d\sigma dr =\int_{\R^3} f(\theta,\sigma,r ,\t) \der \phi(\theta,\sigma,r)\cdot \\
 & \cdot\bigg[ \int_{\R^3} w(r-r')(b(\theta-\theta')+ K(\theta-\theta',\sigma,\sigma')-b(r-r'))f(\theta',\sigma',r' ,\t)\\
 &\qquad \qquad  d\theta' d\sigma' dr'\bigg] d\theta d\sigma dr 
 \end{aligned}
\end{align}
\noindent It remains to show that for suitable boundary conditions equation \eqref{eq:weakeq} gives the desired weak formulation of the Fokker-Planck equation. We calculate
\begin{align*}
\int_{\R} \bigg( & f (\theta,\sigma,r,\t) \phi (\theta,\sigma,r)\big( \int_{\R^3} w(r-r')\cdot  \\
& \cdot (b(\theta-\theta')+ K(\theta-\theta',\sigma,\sigma')-b(r-r'))f(\theta',\sigma',r' ,\t) d\theta' d\sigma' dr' \big) \bigg)_{r=-\infty}^{r=+\infty} dr d\sigma
\end{align*}
This term is zero, if 
\begin{align}\label{eq:a2}
\lim_{|r|\rightarrow +\infty} f(\theta,\sigma,r,\t)=0
\end{align}
These boundary condition are guaranteed for the Boltzmann equation $\fe(\theta,\sigma,r,\t)$ by mass conservation and preservation of the first moment $m_{1,r}$, see \eqref{eq:conm1}.
Then \eqref{eq:weakeq} is the weak form of the Fokker-Planck equation
\begin{align} \label{eq:fokkerplanck}
\frac{d}{d\tau}  \int_{\R^3}  &\phi(\theta,\sigma,r )f(\theta,\sigma,r ,\t) d\theta d\sigma dr = \notag\\
-\int_{\R^3} & \phi(\theta,\sigma,r ) \der\bigg[f(\theta,\sigma,r ,\t) \int_{\R^3} w(r-r')\cdot \\
&\cdot \big(b(\theta-\theta')+ K(\theta-\theta',\sigma,\sigma')-b(r-r'))f(\theta',\sigma',r' \big)  d\theta' d\sigma' dr'\bigg]d\theta d\sigma dr .\notag
\end{align}

\printbibliography

\end{document}